\documentclass[envcountsame, envcountsect, oribibl,plain, english]{amsart}
\usepackage{amsthm}
\usepackage[colorlinks]{hyperref}
\newtheorem{theorem}{Theorem}[section]
\newtheorem{lemma}[theorem]{Lemma}
\newtheorem{corollary}[theorem]{Corollary}

\theoremstyle{remark}
\newtheorem{example}[theorem]{Example}
\newtheorem{remark}[theorem]{Remark}

\theoremstyle{definition}
\newtheorem{definition}[theorem]{Definition}

\newcommand{\abs}[1]{\ensuremath{\left\vert#1\right\vert}}
\newcommand{\kla}[1]{\ensuremath{\left(#1\right)}}
\newcommand{\menge}[1]{\ensuremath{\left\lbrace#1\right\rbrace}}
\newcommand{\scalar}[1]{\ensuremath{\left<#1\right>}}
\newcommand{\norm}[1]{\ensuremath{\left\|#1\right\|}}

\newcommand{\N}{\ensuremath{\mathbb{N}}}
\newcommand{\Z}{\ensuremath{\mathbb{Z}}}

\newcommand{\R}{\ensuremath{\mathbb{R}}}
\newcommand{\C}{\ensuremath{\mathbb{C}}}

\newcommand{\T}{\ensuremath{\mathbb{T}}}

\usepackage{amssymb, times, amsmath, bbold, pgf, tikz, verbatim, stmaryrd, graphicx, geometry, mathpazo}
\usepackage[utf8]{inputenc}
\usepackage[T1]{fontenc}
\usetikzlibrary{arrows}
\usetikzlibrary{matrix}
\usepackage{pst-node}
\usepackage{tikz-cd} 
\usepackage{hyperref}

\title{On the numerical radius of weighted shifts on $\ell^{2}$ as a norm on $\ell^{\infty}$ and connections to Banach limits}
\author{Akram Sharif}

\address{TU Dresden, D-01062 Dresden, Germany \& Mathematisches Institut, Universität Leipzig, D-04109 Leipzig, Germany}

\email{akram.sharif@tu-dresden.de, akram.sharif@math.uni-leipzig.de}


\begin{document}
	
	\maketitle
	\begin{abstract}
		We investigate the numerical radius of weighted shifts on $\ell^{2}$ as a norm on $\ell^{\infty}$. Along the way, we prove that the largest value a real-valued Banach limit can attain on the sequence of absolute values of a bounded sequences is larger or equal to the spectral radius and smaller or equal to the numerical radius of the corresponding weighted shift. Further, we compare the operator norms on $\ell^{\infty}$ with respect to the uniform norm and the numerical radius of weighted shifts as a norm. We show that they generally differ, but for multiplication operators and weighted shifts on $\ell^{\infty}$, they are both equal to the uniform norm of the corresponding weight. Moreover, we prove that all complex-valued Banach limits satisfy the norm inequality with respect to the numerical radius of weighted shifts and provide an application of our results to the theory of Banach limits.
	\end{abstract}
	
	\section{Introduction and Preliminaries}
	
	Weighted shifts are well studied in the literature as they are classical examples of bounded linear operators on Hilbert and Banach spaces, and due to their structural and spectral properties important examples in functional analysis \cite{halmos} and dynamical systems \cite{salas}.  Additionally, they are closely related to numerical analysis and information theory \cite{chen}. The spectral properties of weighted shifts are intensively studied in the literature, as seen in works like \cite{kelley}, \cite{ridge}. Moreover, the numerical range and radius of weighted shifts has been under intensive investigation \cite{chien1, chien2, chien4, ridge2, shields, stout, chien3, WangWu2011}, as the numerical range of weighted shifts is, as well as the spectrum, always a disk centered about the origin. However, determining if the disk is open or closed, as well as computing its radius, is generally non-trivial (for further references, see our list below).
	
	The purpose of this paper is to view the numerical radius of weighted shifts on $\ell^{2}$ as a norm on $\ell^{\infty}$ and investigate among other things it’s relation to the uniform norm and it's connection to the area of Banach limits, i.e. normalized, shift-invariant, norm-preserving linear functionals, which extend the limit functional from the space of convergent sequences to $\ell^{\infty}$. One of our main results is that for any complex-valued bounded sequence $a\in\ell^{\infty}$, we have 
	\begin{align}
		\label{eq: sharpened connection intro}
		r(T_{a}) \leq \lim_{n\to\infty} \sup_{j\in\N} \frac{1}{n} \sum_{k=1}^{n} \abs{a_{k+j-1}} \leq w(T_{a}),
	\end{align}
	where $r(T_{a})$ is the spectral and $w(T_{a})$ is the numerical radius of the weighted shift $T_{a}$ corresponding to the weight $a$, and the expression in between is according to Sucheston \cite{Sucheston} the largest value a real-valued Banach limit can attain on the sequence of absolute values of a bounded sequences. 
	
	We denote by $\mathcal{B}(X)$ the algebra of all bounded linear operators on a complex Banach space $X$. Now let $H$ be a complex Hilbert space. Then the numerical range of $T\in\mathcal{B}(H)$ is defined to be the set
	\begin{align*}
		W(T) := \menge{\scalar{Tx,x}\,:\,x\in H,\,\norm{x}=1}.
	\end{align*}
	It is well known that $W(T)$ is bounded and convex (cf. \cite[Theorem 1.1-2]{GustafsonRao1997}). The numerical radius of $T$ is defined as
	\begin{align*}
		w(T) := \sup\menge{\abs{\scalar{Tx,x}}\,:\,x\in H,\,\norm{x}=1} = \sup\menge{\abs{\lambda}\,:\,\lambda\in W(T)}.
	\end{align*}
	Moreover one has the following well known properties for the numerical range and radius:
	\begin{enumerate}
		\item[1.] $w(T)$ is a norm on $\mathcal{B}(H)$ and for each $T\in\mathcal{B}(H)$ we have \begin{align}
			\label{ineq: equiv of op norms}
			w(T)\leq \norm{T}\leq 2w(T)
		\end{align}
		And if $T$ is hyponormal, that is if $T^{*}T-TT^{*}$ is positive definite, then $w(T) = \norm{T}$.
		\item[2.] $\sigma(T)\subseteq\overline{W(T)}$
		\item[3.] $W(T)$ is compact, if $T$ is compact
		\item[4.] $W(U^{*}TU) = W(T)$ for all unitary operators $U\in\mathcal{B}(H)$
		\item[5.] $W(\Re(T)) = \menge{\mathrm{Re}(z)\,:\,z\in W(T)}$, where $\Re(T) = \frac{1}{2}(T+T^{*})$ is the hermitian real part of $T$
	\end{enumerate}
	
	For references, proofs and more context related to the numerical range we refer to \cite[Chapters 1.1-1.5 and Chapter 6.2]{GustafsonRao1997}.
	
	\begin{definition}
		Let $T\in\mathcal{B}(H)$. Then $T$ is called 
		\begin{enumerate}
			\item[1.] \textit{normaloid}, if $w(T) = \norm{T}$.
			\item[2.] \textit{convexoid}, if $\overline{W(T)} = \mathrm{conv} \sigma(T)$.
		\end{enumerate}
	\end{definition}
	Note that we have the following implications for $T\in\mathcal{B}(H)$ \cite[Chapter 6.1 - 6.2]{GustafsonRao1997}
	\begin{align}
		\label{impl normaloid}
		T\mbox{ hyponormal}  \quad \Rightarrow \quad \forall \lambda\in\C:\,\,\,T-\lambda I \mbox{ normaloid} \quad \Rightarrow  \quad T \mbox{ convexoid} 
	\end{align}
	
	In what follows we let $\omega$ be the vector space of all complex-valued sequences, $\ell^{2}$ be the complex Hilbert space of square summable unilateral sequences and $\ell^{\infty}$ the complex Banach space of bounded sequences respectively.	If $a,b\in\ell^{\infty}$, then we denote by $ab$ the product sequence $ab := \kla{a_{k}b_{k}} \in\ell^{\infty}$. Moreover, we set $\mathbb{1}:=(1,1,...)\in\ell^{\infty}$, and for any sequence $a=\kla{a_{k}}\in\ell^{\infty}$, we denote by $\abs{a}$ the sequence of its absolute values, i.e. $\abs{a} := \kla{\abs{a_{k}}}\in\ell^{\infty}$. To ease our exposition, we also introduce the notation $a\leq b$ for any two real-valued sequences $a,b\in\ell^{\infty}$ that satisfy $a_{k} \leq b_{k}$ for all $k\in\N$. Finally, for a periodic sequence $a\in\ell^{\infty}$ with period $N\in\N$, we clearly have $a\in\ell^{\infty}$ with $\norm{a}_{\infty} = \max\menge{\abs{a_{1}},..., \abs{a_{N}}}$ and we denote the arithmetic mean of $a$ by   
	\begin{align*}
		\overline{a} := \frac{1}{N}\sum_{k=1}^{N} a_{k}.
	\end{align*}

	\begin{definition}
		Let $a=(a_{k})\in\omega$. Then, we define by 
		\begin{align*}
			T_{a}:\omega\to\omega,\,\kla{x_{k}}\mapsto\kla{0,a_{1}x_{1},a_{2}x_{2},...}
		\end{align*}
		the \textit{(unilateral) weighted (forward) shift} associated to the sequence $a$. If $a$ is periodic of period $N\in\N$ we sometimes also write $T(a_{1},...,a_{N})$ instead of $T_{a}$.  
		\\Further we simple denote by $T:\ell^{2}\rightarrow \ell^{2}$ the (unilateral) forward-shift operator and by $S:\ell^{2}\to\ell^{2}$ the (unilateral) backward-shift.
	\end{definition}
	
	In this paper we will only focus on unilateral forward shifts. However results we mention are partially transferable to weighted backward shifts and into the bilateral setting. For connections to weighted shift matrices we refer to \cite{chien3} and \cite{WangWu2011} as well as the references therein.
	
	Let us list some well-known facts about weighted shifts, their spectrum as well as their numerical range and radius: 
	\begin{enumerate}
		\item[1.]  $T_{a}\in\mathcal{B}\kla{\ell^{2}}$ if and only if $a\in\ell^{\infty}$. Hence we assume from now on always $a\in\ell^{\infty}$.
		\item[2.]  $\norm{T_{a}} = \norm{a}_{\infty}$ and more generally for $n\in\N$ we have
		\begin{align*}
			\norm{T_{a}^{n}} = \sup_{j\in\N} \,\prod_{k=1}^{n} \abs{a_{k+j-1}}.
		\end{align*}
		\item[3.] Kelley \cite{kelley} proved that $\sigma(T_{a})$ is a closed circular disk about the origin, with spectral radius
		\begin{align}
			\label{eq: spectral radius}
			r\kla{T_{a}}= \lim_{n\to\infty} \norm{T_{a}^{n}}^{\frac{1}{n}}  = \lim_{n\to\infty}\sup_{j\in\N}\kla{\prod_{k=1}^{n}\abs{a_{k+j-1}}}^{\frac{1}{n}}.
		\end{align}
		
		\item[4.]  Shields \cite{shields} has shown that two weighted shifts $T_{a}$, $T_{b}$ are unitarily equivalent if and only if $\abs{a} = \abs{b}$, and in particular for any $\theta\in\R$ a weigthed shift $T_{a}$ is unitarily equivalent to $e^{i\theta} T_{a}$. This shows that it suffices to consider shifts with real non-negative weights and by the invariance of the numerical range under unitary conjugation implies that the numerical range  the numerical range of weighted shifts is a disk about the origin. So the analysis of the numerical range boils down to the following problems 
		\begin{enumerate}
			\item[(P1)] Determine whether $W(T_{a})$ an open or closed disk
			\item[(P2)] Find the exact value of $w(T_{a})$
		\end{enumerate} 
		
		\item[5.] Berger and Stampfli \cite{BergerStampfli1967} computed the numerical range and radius of a weighted shift with weights $(r,1,1,1,...)$ for $r\geq0$. They showed that 
		\begin{align*}
			w(T_{(r,1,1,…)}) = \begin{cases}
				1, & \mbox{if }0\leq r\leq \sqrt{2},\\
				\frac{1}{2}\kla{\kla{r^2 - 1}^{\frac{1}{2}} + \kla{r^2-1}^{-\frac{1}{2}}}, & \mbox{if } \sqrt{2} < r.
			\end{cases} 
		\end{align*} 
		
		\item[6.] Eckstein and Rasz \cite{eckstein} characterized by necessary and sufficient conditions all weights $a\in\ell^{\infty}$, which define numerical radius contractions, i.e. satisfy $w(T_{a})\leq 1$.
		
		\item[7.] Ridge \cite[Theorem 1]{ridge} proved that for $a\in\ell^{\infty}$ periodic of Period $N\in\N$ one has
		\begin{align}
			\label{eq: Ridge}
			w(T_{a}) = \max\menge{\abs{a_{1}}x_{1}x_{2} + ... + \abs{a_{N}}x_{N}x_{1}\,:\,x\in\R^{N},\,\norm{x}_{2}^{2} = 1}
		\end{align}
		and conjectured in the same paper that $W(T_{a})$ is open, if all entries of $a$ are non-zero. 
		
		\item[8.] Stout \cite{stout} proved the conjecture of Ridge and observed further that it is possible for weighted shifts of Hilbert-Schmidt class, i.e. $a\in\ell^{2}$, to replace the variational problem 
		\begin{align*}
			w(T_{a})  & = \sup\menge{\abs{\scalar{T_{a}x,x}}\,:\,x\in\ell^{2},\,\norm{x}_{\ell^{2}}=1} \\
			& = \sup\menge{\abs{\sum_{k=1}^{\infty} a_{k}x_{k}x_{k+1}^{*}}\,:\;x\in\ell^{2},\,\sum_{k=1}^{\infty} \abs{x_{k}}^{2}=1}.
		\end{align*} 
		by an algebraical problem. To be more precise, consider for a moment without loss of generality any non-negative $a\in\ell^{\infty}$. Since the numerical range of a weighted shift is a disk about the origin, one finds the equality $w(T_{a}) = w\kla{\Re(T_{a})}$. Notice that $\Re(T_{a})$ is a normal operator, so one has $w\kla{\Re\kla{T_{a}}} = \norm{\Re\kla{T_{a}}}$. Now if one additionally assumes $a\in\ell^{2}$, it is possible to show that $\norm{\Re\kla{T_{a}}} = \lambda^{-1}$, where $\lambda$ is the minimal positive root of the analytic function
		\begin{align*}
			F_{a}(z) = \det\kla{I-z\Re\kla{T_{a}}},
		\end{align*}
		which can be expressed explicitly as in \cite[Theorem 3, p. 497]{stout} (or see for a summarization \cite[p. 945]{chien3}).
	\end{enumerate}
	
	As the previous list suggests it is fairly complicated to give answers to problems (P1) and (P2) for general weights. However, there are numerous contributions addressing these problems, providing solutions, characterizations or partial answers for certain weights:
	\begin{enumerate}
		\item Chien and Nakazato \cite{chien1} computed an upper and lower bound for the numerical radius of weighted shifts with geometric weights $(1,r,r^{2},...)$ for $0<r<1$ and showed that the numerical range of the associated weighted shift is closed
		\item Wang and Wu \cite{WangWu2011} contributed among other things Perron-Frobenius type results for the numerical radius of weighted shifts, a new proof and refinements of the Eckstein-Rasz Theorem in order to give some answers to problem (P1) and (P2)
		\item Computations, estimates and characterizations of radii of weighted shifts with weights such as $(r,1,1,…)$, $(1,r,1,1,…)$, $(1,…,1,r,1,1,…)$ with $r$ at the $m$-th position, $(r,r^{2},r^{3},…)$ as well as $(h,k,a,b,a,b,...)$ and more generally $(h_{1},...,h_{m},a,b,a,b...)$ are treated in \cite{cob1, cob2, chien1, chien2, chien4, Lasser/Obermaier, tam, UndarkhVandanjav2014, chien3, WangWu2011}.
	\end{enumerate} 
	
	Next, we want to outline the role of the numerical radius of weighted shifts as a norm on $\ell^{\infty}$. More precisely note that
	\begin{align*}
		\Phi:\ell^{\infty}\rightarrow \mathcal{B}(\ell^{2}),\,a\mapsto T_{a}
	\end{align*}
	is a linear isometry, which allows us to identify $\ell^{\infty}$ with the image of $\Phi$ and set $\norm{a}_{w} := w(T_{a}) = w(\Phi(a))$ for all $a\in\ell^{\infty}$. Thus the map $\norm{\cdot}_{w}:\ell^{\infty}\to \left[0,\infty\right)$ is a norm on $\ell^{\infty}$ and by inequality \eqref{ineq: equiv of op norms} we obtain for all $a\in\ell^{\infty}$ 
	\begin{align}
		\label{ineq: equiv of norms}
		\norm{a}_{w} \leq \norm{T_{a}} = \norm{a}_{\infty} \leq 2\norm{a}_{w}.
	\end{align}
	Hence $\norm{\cdot}_{w}$ is equivalent to the uniform norm on $\ell^{\infty}$. 
	\begin{definition}
		A norm $\norm{\cdot}\colon\ell^{\infty}\to[0,\infty)$ is called 
		\begin{enumerate}
			\item[1.] \textit{absolute}, if
			\begin{align*}
				\forall a,b\in\ell^{\infty}:\qquad \abs{a} = \abs{b} \quad \Rightarrow \norm{a}_{\infty} = \norm{b}_{\infty}.
			\end{align*}
			
			\item[2.] \textit{monotonic}, if 
			\begin{align*}
				\forall a,b\in\ell^{\infty}:\qquad \abs{a} \leq \abs{b} \quad \Rightarrow \norm{a}_{\infty} \leq \norm{b}_{\infty}.
			\end{align*}
		\end{enumerate}
	\end{definition}
	
	Note that on $\ell^{\infty}$ any monotonic norm is absolute. The opposite is also true in finite dimensions, i.e. on $\C^{n}$ any norm is monotonic if and only if it is absolute. The author does not know if this is true on $\ell^{\infty}$.
	
	Obviously the uniform norm is absolute and monotonic. Further by Shields \cite{shields} result it is also clear that $\norm{\cdot}_{w}$ is absolute. Using this fact it is easy to show that $\norm{\cdot}_{w}$ is monotonic \cite[Prop. 2.5 (a)]{WangWu2011}. Wang and Wu further proved \cite[Lemma 4.4 (c)]{WangWu2011} that for periodic sequences $a,b\in\ell^{\infty}$ with the same period $N\in\N$ and such that there exists a number $n_{0}\in\N$ with $\abs{a_{n_{0}}}< \abs{b_{n_{0}}}$ one has
	\begin{align*}
		\norm{a}_{w}< \norm{b}_{w},
	\end{align*}
	which is clearly a property that the uniform norm also has.
	
	\begin{definition}
		Let $R:\ell^{\infty}\rightarrow\ell^{\infty}$ be linear. Then we define the operator norm of $R$ with respect to $\norm{\cdot}_{w}$ as
		\begin{align*}
			\norm{R}_{w} := \inf\menge{c>0\,:\,\norm{Ra}_{w} \leq c\cdot \norm{a}_{w},\,\,\forall a\in\ell^{\infty}} = \sup_{\norm{a}_{w} = 1} \norm{Ra}_{w}.
		\end{align*}
		Moreover we define for a linear map $\phi:\ell^{\infty}\rightarrow \C$ the operator norm of $\phi$ with respect to $\norm{\cdot}_{w}$ as
		\begin{align*}
			\norm{\phi}_{w} := \sup_{\norm{a}_{w}=1} \abs{\phi(a)} = \sup_{\norm{a}_{w}\neq 0} \frac{\abs{\phi(a)}}{\norm{a}_{w}}
		\end{align*}
	\end{definition}
	
	As a direct consequence of this definition and \eqref{ineq: equiv of norms}, i.e. the equivalence of the uniform norm and $\norm{\cdot}_{w}$, we conclude
	\begin{enumerate}
		\item[1.] for all $R\in\mathcal{B}\kla{\ell^{\infty}}$ we have
		\begin{align}
			\label{ineq: op norm}
			\frac{\norm{R}}{2} \leq \norm{R}_{w} \leq 2\norm{R}
		\end{align}
		\item[2.] for all $\phi\in\kla{\ell^{\infty}}^{*}$
		\begin{align}
			\label{ineq: op norm 2}
			\norm{\phi}\leq \norm{\phi}_{w}\leq 2\norm{\phi}.
		\end{align}
	\end{enumerate}
	
	\begin{remark} Let $R\in\ell^{\infty}$. If we equip $\ell^{\infty}$ with mixed norms, we define the operator norm of $R$ as  
		\begin{align*}
			\norm{R}_{\infty,w} := \sup_{\norm{a}_{\infty}=1} \norm{Ra}_{w},\quad   \norm{R}_{w,\infty} := \sup_{\norm{a}_{w}=1} \norm{Ra}_{\infty}.
		\end{align*}
		In this case we conclude 
		\begin{align*}
			\frac{\norm{R}}{2} \leq \norm{R}_{\infty,w} \leq \norm{R}_{w} \leq \norm{R}_{w,\infty} \leq 2\norm{R},
		\end{align*}
		since we have for any $a\in\ell^{\infty}$
		\begin{align*}
			\norm{Ra}_{w} \leq  \norm{R}_{w} \norm{a}_{w} \leq \norm{R}_{w} \norm{a}_{\infty}, \quad  \frac{1}{2}\norm{Ra}_{\infty} \leq \norm{Ra}_{w}
		\end{align*}
		and
		\begin{align*}
			\norm{Ra}_{\infty} \leq \norm{R}\norm{a}_{\infty} \leq 2\norm{R} \norm{a}_{w},\quad \norm{Ra}_{w} \leq \norm{Ra}_{\infty}.
		\end{align*}
	\end{remark}
	
	Of course one can also adjust the idea of using the numerical radius of weighted shifts as a norm, to define a norm on $\C^{N}$ as the following example shows.
	
	\begin{example}
		Let $N\in\N$. For any $a\in \C^{N}$ we define $\norm{a}_{w,N} := w(T(a_{1},...,a_{N}))$. Then $\norm{\cdot}_{w,N}$ is a norm on $\C^{N}$. It is then obviously equivalent to the uniform norm, since all norms on $\C^{n}$ are equivalent and we have
		\begin{align*}
			\norm{a}_{w,N} \leq \max\menge{\abs{a_{1}},...,\abs{a_{N}}}\leq 2\norm{a}_{w,N},\quad \forall a\in \C^{N}. 
		\end{align*}
		Note that we have for $N=1$ and $a\in \C$
		\begin{align*}
			\norm{a}_{w,1} = \abs{a}
		\end{align*}
		and for $N=2$ and $a=(a_{1},a_{2})\in\C^{2}$
		\begin{align*}
			\norm{a}_{w,2} = \frac{\abs{a_{1}}+\abs{a_{2}}}{2}.
		\end{align*}
		For arbitrary $N\in\N$, one may use the result of Ridge, i.e. equality \eqref{eq: Ridge}, to conclude for any $a=(a_{1},...,a_{N})\in\C^{N}$ 
		\begin{align*}
			\norm{a}_{w,N} = \max\menge{\abs{a_{1}}x_{1}x_{2} + ... + \abs{a_{N}}x_{N}x_{1}\,:\,x\in\R^{N},\,\norm{x}_{2}^{2} = 1}.
		\end{align*}
	\end{example}
	
	The rest of this work is organized as follows. In Section 2 we prove that the largest value a Banach limit can attain on a bounded sequence is a lower bound for $\norm{\cdot}_{w}$ and conclude \eqref{eq: sharpened connection intro} from this. Further, we specialize the bound for certain weights, e.g. Cesàro convergent weights. We then give some examples on how these bounds may give some useful information and show that the inequalities in \eqref{ineq: op norm} and \eqref{ineq: op norm 2} can be strict and the constants $\frac{1}{2}$ and $2$ are optimal (see Examples \ref{ex: norm of op} and Example \ref{ex: projection}). Further we generalize a result by Tam \cite[Theorem 1 (a)]{tam} and also show that non-constant periodic shifts are non-convexoid. Section 3 deals with the operator norm of multiplication operators and weighted shifts on $\ell^{\infty}$ with respect to $\norm{\cdot}_{w}$. In Section 4 we prove that all complex-valued Banach limits satisfy $\norm{L}_{w} = 1$ by combining Sucheston's Theorem \cite[Theorem, p. 309]{Sucheston} on the maximal and minimal value of real valued Banach limits on a sequence and the lower bound from Section 2. Further, we discuss one application of the connection between the maximal value Banach limit can attain on the sequence of absolute values of a bounded sequence and the spectral and numerical radius given in \eqref{eq: sharpened connection}. Finally, Section 5 is devoted to give sufficient conditions for other norms on $\ell^{\infty}$ to also have the property, that all complex-valued Banach limits have operator norm equal to 1 with respect to the norms.
	\\\\\textbf{Acknowledgement:}  The author sincerely thanks his diploma supervisor Agnes Radl for introducing him to the numerical range of bounded operators, which laid the the foundation for this work. Further, the author would like to express his sincere gratitude to Alexandr Usachev for his interest in this work and bringing the paper of Sucheston to his attention, which helped not only solving a question from an earlier version of this work, but also improved the main results of this paper.
	Moreover, the author thanks Martin Tautenhahn for useful discussions.
	Lastly, the author would like to thank the Hausdorff Institute for Mathematics in Bonn (Germany), for providing an outstanding research environment during his visit. This research was partially funded by the Deutsche Forschungsgemeinschaft (DFG, German Research Foundation) under Germany's Excellence Strategy – EXC-2047/1 – 390685813.

	\section{An estimate from below and consequences}
	
	As we pointed out in the introduction, it is quite difficult to compute the numerical radius of weighted shifts for an arbitrary weight. To the best knowledge of the author there are also no general estimates from below known in the literature for arbitrary weights. The next theorem provides such a general estimate from below, by using ideas for Banach limits from Sucheston \cite{Sucheston}.
	
	\begin{theorem}
		\label{th: estimates}
		Let $a\in\ell^{\infty}$. Then 
		\begin{align}
			\label{eq: estimate any seq}
			\lim_{n\to\infty}\sup_{j\in\N}  \frac{1}{n} \sum_{k=1}^{n} \abs{a_{k+j-1}} \leq \norm{a}_{w}.
		\end{align}
	\end{theorem}
	\begin{proof}
		Fix arbitrary $j\in\N$ and $n\in\N$ for the moment. Then for any $x\in\ell^{2}$, we have
		\begin{align*}
			\scalar{T_{a}x,x}_{\ell^{2}} = \sum_{k=1}^{\infty} a_{k}x_{k}x_{k+1}^{*}.
		\end{align*}
		Hence we find for each $x\in\ell^{2}$ with $\norm{x}_{\ell^{2}} = 1$ 
		\begin{align*}
			\sum_{k=1}^{\infty} \abs{a_{k}}x_{k}x_{k+1}^{*} = \scalar{T_{\abs{a}}x,x}_{\ell^{2}} \leq  \norm{\abs{a}}_{w} = \norm{a}_{w}  
		\end{align*}
		For each $n\in\N$ we choose $x_{n} := \frac{1}{\sqrt{n}}\sum\limits_{k=1}^{n} e_{k+j-1}$, where $e_{k}$ is the $k$-th unit vector in $\ell^{2}$. Then for all $n\in\N$, we have $x_{n}\in\ell^{2}$, $\norm{x_{n}}_{\ell^{2}} = 1$ and
		\begin{align}
			\label{ineq: estimate nr}
			\frac{1}{n} \sum_{k=1}^{n-1} \abs{a_{k+j-1}}\leq \norm{a}_{w}.
		\end{align}
		Next, we consider for each $n\in\N$ the linear map $C_{n}:\ell^{\infty}\to\ell^{\infty}$ given by
		\begin{align*}
			C_{n}a := \kla{\frac{1}{n} \sum_{k=1}^{n} a_{k+j-1} }_{j\in\N},\quad a\in\ell^{\infty}.
		\end{align*}
		Clearly $C_{n}\in\mathcal{B}(\ell^{\infty})$ for all $n\in\N$ and choosing $n\in\N$ large enough and $j\in\N$, we have by \eqref{ineq: estimate nr}
		\begin{align*}
			\frac{1}{n} \sum_{k=1}^{n} \abs{a_{k+j-1}} \overset{\eqref{ineq: estimate nr}}{\leq} \norm{a}_{w} + \frac{\abs{a_{n+j-1}}}{n} \leq \norm{a}_{w} + \frac{\norm{a}_{\infty}}{n},   
		\end{align*}
		from which we obtain
		\begin{align*}
			c_{n}:=\norm{C_{n}\abs{a}}_{\infty} \leq \norm{a}_{w} + \frac{\norm{a}_{\infty}}{n}.
		\end{align*}
		Finally, observe that by the means of \cite[p. 309]{Sucheston} the sequence $\kla{c_{n}}$ converges, which implies
		\begin{align*}
			\lim_{n\to\infty}\sup_{j\in\N}  \frac{1}{n} \sum_{k=1}^{n} \abs{a_{k+j-1}} = \lim_{n\to\infty} c_{n} \leq \norm{a}_{w}.
		\end{align*}
	\end{proof}
	Now let us define for any $a\in\ell^{\infty}$
	\begin{align*}
		M(\abs{a}) :=  \lim_{n\to\infty}\sup_{j\in\N}  \frac{1}{n} \sum_{k=1}^{n} \abs{a_{k+j-1}}.
	\end{align*}
	Note that by Sucheston's Theorem \cite[Theorem, p. 309]{Sucheston} the number $M(\abs{a})$ is the largest value a real-valued Banach limit (for a definition see, Definition \ref{def:banach limit}) attains on the sequence of absolute values of $a\in\ell^{\infty}$.
	Further, the inequality of arithmetic and geometric means yields for any $n\in\N$ and $j\in\N$
	\begin{align*}
		\kla{\prod_{k=1}^{n} \abs{a_{k+j-1}}}^{\frac{1}{n}} \leq \frac{1}{n}\sum_{k=1}^{n} \abs{a_{k+j-1}}, \quad \forall a\in\ell^{\infty}. 
	\end{align*}
	Thus for all $a\in\ell^{\infty}$, we conclude 
	\begin{align}
		r(T_{a}) = \lim_{n\to\infty}\sup_{j\in\N}\kla{\prod_{k=1}^{n}\abs{a_{k+j-1}}}^{\frac{1}{n}} \leq \lim_{n\to\infty}\sup_{j\in\N}  \frac{1}{n} \sum_{k=1}^{n} \abs{a_{k+j-1}} = M(\abs{a}),
	\end{align}
	which implies for all $a\in\ell^{\infty}$
	\begin{align*}
		r(T_{a}) \leq M(\abs{a}) \leq \norm{a}_{\infty} = \norm{T_{a}}.
	\end{align*}
	This provides a connection between the maximal value of Banach limits on the sequence $\abs{a}$ for a given sequence $a\in\ell^{\infty}$ (cf. \cite[Theorem, p. 309]{Sucheston}) with the spectral radius and norm of the weighted shift corresponding to $a$. Notice that inequality \eqref{eq: estimate any seq} sharpens this connection between Banach limits and weighted shifts as follows:
	\begin{corollary}
		For all $a\in\ell^{\infty}$, we have
		\begin{align}
			\label{eq: sharpened connection}
			r(T_{a}) \leq M(\abs{a}) \leq \norm{a}_{w} \leq \norm{a}_{\infty}.
		\end{align}
	\end{corollary}
	Hence, we have proven \eqref{eq: sharpened connection intro} and there is actually the even stronger connection between the maximal value of real-valued Banach limits on the sequence $\abs{a}$ for a given $a\in\ell^{\infty}$ with the spectral- and numerical radius of the weighted shift corresponding to $a$. The next example shows that any of the above inequalities can be strict:
	\begin{example}
		Let $a\in\ell^{\infty}$ be periodic sequence of period $N=3$ with $a_{1}=a_{3}=\frac{2}{3}$ and $a_{2}=1$. Then
		\begin{align*}
			& r(T_{a}) = \sqrt[3]{a_{1}a_{2}a_{3}} = \kla{\frac{2}{3}}^{\frac{2}{3}} \approx 0.7631 \\ 
			& M(\abs{a}) = \frac{a_{1}+a_{2}+a_{3}}{3} = \frac{7}{9} \approx 0.7777 \\
			& \norm{a}_{w} = \frac{3+\sqrt{41}}{12} \approx 0.7836 \\
			& \norm{a}_{\infty} = 1.
		\end{align*}
		Here we used, that whenever $a\in\ell^{\infty}$ is a periodic sequence of period $N\in\N$, we have
		\begin{align*}
			& r(T_{a}) = \kla{\prod_{k=1}^{N} a_{k}}^{\frac{1}{N}} \\
			& M(\abs{a}) = \overline{\abs{a}} = \frac{1}{N} \sum_{k=1}^{N} \abs{a_{k}}
		\end{align*}
		and $\norm{a}_{w}$ can be computed from \eqref{eq: Ridge} (see  Ridge's Theorem \cite[Theorem 1, p. 107]{ridge2} or \cite[Example (2), p. 110]{ridge2}).
	\end{example}

	Now let us turn to further consequences and applications of inequality \eqref{eq: estimate any seq}. The  next result illustrates several specializations of it, i.e. the case of Cesáro convergent, convergent and periodic sequences.
	
	\begin{corollary}
		\label{cr: estimates}
		Let $a\in\ell^{\infty}$. Then the following hold
		\begin{enumerate}
			\item[(a)] if $|a|$ is Cesáro convergent, we have
			\begin{align*}
				\lim_{n\to\infty} \frac{1}{n} \sum_{k=1}^{n} \abs{a_{k}} \leq \norm{a}_{w}.
			\end{align*}
			\item[(b)] if $\abs{a}$ is convergent, we have
			\begin{align}
				\label{ineq: convergent}
				\lim_{k\to\infty}\abs{a_{k}}\leq \norm{a}_{w}.
			\end{align}
			\item[(c)] if $a$ is periodic of period $N\in\N$, we have
			\begin{align}
				\label{ineq: arithm}
				\overline{\abs{a}} = \frac{1}{N}\sum_{k=1}^{N} \abs{a_{k}} \leq \norm{a}_{w}.
			\end{align}
			\item[(d)] if there exists a periodic sequence $a_{1}\in\ell^{\infty}$ and a convergent sequence $a_{2}$ such that $a=a_{1}+a_{2}$, we have
			\begin{align}
				\label{eq: conv plus per}
				\abs{\overline{a_{1}}+\lim_{k\to\infty} a_{2,k}} \leq \norm{a}_{w}.
			\end{align}
		\end{enumerate}
	\end{corollary}
	\begin{proof}
		From the proof of the previous item it follows that for any $n\in\N$, we have
		\begin{align*}
			\frac{1}{n}\sum_{k=1}^{n} \abs{a_{k}} \leq \norm{C_{n}\abs{a}}_{\infty}
		\end{align*}
		Thus, we obtain
		\begin{align*}
			\limsup_{n\to\infty} \frac{1}{n}\sum_{k=1}^{n} \abs{a_{k}} \leq  \lim_{n\to\infty}\sup_{j\in\N}  \frac{1}{n} \sum_{k=1}^{n} \abs{a_{k+j-1}} \leq \norm{a}_{w}.
		\end{align*}
		Hence (a) follows immediately and (b), (c) and (d) are consequences of (b).
	\end{proof}
	
	\begin{remark}
		\label{rm: limit operator}
		\begin{enumerate}
			\item[1.] The lower bound \eqref{ineq: arithm} in the periodic case, is clearly sharp for constant sequences. Further it is also attained in the case of 2-periodic sequences (for a proof see \cite[Example (1)]{ridge2}), i.e. we have 
			\begin{align*}
				\norm{a}_{w} = \frac{\abs{a_{1}}+\abs{a_{2}}}{2} = \overline{\abs{a}}
			\end{align*}
			
			\item[2.] Note that one can also recover inequality \eqref{ineq: arithm} from equation \eqref{eq: Ridge} choosing $x\in\R^{N}$ to be the unit vector with all entries equal to $\frac{1}{\sqrt{N}}$.
			
			\item[3.] Inequality \eqref{ineq: convergent} was first proven by Wang and Wu \cite[Prop. 2.2 (b)]{WangWu2011} using a different method involving the essential numerical range. We would like to point out, that it is possible to prove
			\begin{align*}
				\limsup_{n\to\infty} \frac{1}{n}\sum_{k=1}^{n} \abs{a_{k}} \leq M(\abs{a}) \leq \limsup_{n\to\infty} \abs{a_{n}},\quad \forall a\in\ell^{\infty}.
			\end{align*}
			Thus, if $a\in\ell^{\infty}$ and $|a|$ is convergent, we obtain
			\begin{align*}
				\lim_{n\to\infty}\sup_{j\in\N}  \frac{1}{n} \sum_{k=1}^{n} \abs{a_{k+j-1}} = \lim_{k\to\infty} \abs{a_{k}}.
			\end{align*}
			Hence, inequality \eqref{eq: estimate any seq} reduces  in the case that $\abs{a}$ converges to inequality \eqref{ineq: convergent}. Further, inequality \eqref{ineq: convergent} actually shows that the limit operator 
			\begin{align*}
				\lim:c\to\C, \,\kla{a_{k}} \mapsto \lim_{k\to\infty} a_{k} 
			\end{align*}
			is continuous on the Banach space $(c,\norm{\cdot}_{w})$ of complex-valued convergent sequences equipped with the norm $\norm{\cdot}_{w}$ and has operator norm equal to 1, where the equality follows from considering the sequence $a=\mathbb{1} = \kla{1,1,…}$.
			
			\item[4.] Note that in case (d) of the previous Corollary, inequality \eqref{eq: conv plus per} actually provides a bound for any finite linear combination of convergent sequences and periodic sequences, as the linear combination of finitely many convergent sequences is convergent and similarly the linear combination of finitely many periodic sequences is periodic.
			
			\item[5.] Finally, we would like to point out, that the uniform norm satisfies for any $a\in\ell^{\infty}$ the inequality
			\begin{align*}
				\limsup_{k\to\infty} \abs{a_{k}} \leq \norm{a}_{\infty}.
			\end{align*}
			However $\norm{\cdot}_{w}$ does not satisfy such an inequality in general as we have on the one hand for any convergent sequence $a\in\ell^{\infty}$
			\begin{align*}
				\limsup_{k\to\infty} \abs{a_{k}} = \lim_{k\to\infty} \abs{a_{k}} \leq \norm{a}_{w},
			\end{align*}
			but on the other hand for any 2-periodic sequence $a\in\ell^{\infty}$
			\begin{align*}
				\norm{a}_{w} = \frac{\abs{a_{1}}+\abs{a_{2}}}{2} \leq \norm{a}_{\infty} = \max\menge{\abs{a_{1}},\abs{a_{2}}} = \limsup_{k\to\infty} \abs{a_{k}}.
			\end{align*}
			Note that both inequalities can be strict, e.g. for the convergent sequence $a=\kla{\frac{1}{n}}_{n\in\N}$ and periodic sequence $a=(0,1,0,1,…)$.
		\end{enumerate}
	\end{remark}
	
	We will now look at some examples, where we use the previous inequalities to give first estimates from below for the numerical radius of weighted shifts for certain weights. 
	
	\begin{example}
		We consider the sequence $a=\kla{a_{n}}\in\ell^{\infty}$ defined by 
		\begin{align*}
			a_{n} := \begin{cases}
				2-\frac{1}{n}, & \mbox{if $n$ is odd} \\
				1+\frac{1}{n}, & \mbox{else}. 
			\end{cases}
		\end{align*}
		Then $a$ is the sum of the convergent sequence $d:=\kla{\frac{(-1)^{n}}{n}}\in c_{0}$ and the periodic sequence $b = \kla{1,2,1,2,...}$, so that we obtain the bounds 
		\begin{align*}
			\frac{3}{2} \overset{\eqref{eq: conv plus per} }{\leq} \norm{a}_{w} \leq \norm{a}_{\infty} = 2.
		\end{align*}
	\end{example}
	
	\begin{example}
		We consider a sequence of the type $a=(h,k,a_{1},a_{2},a_{1},a_{2},...)\in\ell^{ \infty}$, where $h,k,a_{1},a_{2}\in\C$. Such a weight is a sum of the convergent sequence $(h-a_{1},k-a_{2},0,0,...)$ converging to zero and the 2-periodic sequence $(a_{1},a_{2},a_{1},a_{2},...)$. Thus we may apply inequality \eqref{eq: conv plus per} to obtain the following estimate
		\begin{align*}
			\frac{\abs{a_{1}}+\abs{a_{2}}}{2} \leq \norm{a}_{w}. 
		\end{align*}   
		Moreover in the case $\abs{h}\leq \abs{a_{1}}$ and $\abs{k}\leq \abs{a_{2}}$, we find by monotonicity of the numerical radius of weighted shifts, that
		\begin{align*}
			\norm{a}_{w} = \frac{\abs{a_{1}}+\abs{a_{2}}}{2}.
		\end{align*} 
		
	\end{example}
	
	\begin{example}
		\label{ex: norm of op}
		In this example we show, that there exist an operator $R\in\mathcal{B}(\ell^{\infty})$ with $\norm{R}=1$ and $\norm{R}_{w}=2$. To this end we consider for each $m\in\N$ the bounded linear operator $R_{m}:\ell^{\infty}\to\ell^{\infty}$ given by
		\begin{align*}
			R_{m}a := \kla{a_{m},a_{m},\dots} = a_{m}\cdot \mathbb{1},\quad \forall a\in\ell^{\infty}.
		\end{align*}
		Hence, we have for all $a\in\ell^{\infty}$
		\begin{align*}
			\norm{R_{m}a}_{\infty} = \abs{a_{m}} = \norm{R_{m}a}_{w}.
		\end{align*}
		This readily shows $\norm{R_{m}} = 1$ and
		\begin{align*}
			\norm{R_{m}a}_{w} = \abs{a_{m}} = 2\cdot\frac{\abs{a_{m}}}{2} \overset{\eqref{ineq: estimate nr}}{\leq} 2\norm{a}_{w}.
		\end{align*}
		Finally, if $m$ is odd the sequence $a=(2,0,2,0,\dots)$ satisfies $\norm{a}_{w} = 1$ and $\norm{R_{m}a}_{w} = 2$. If $m$ is even, the sequence $a=(0,2,0,2,\dots)$ yields the same result. This shows that $\norm{R_{m}}_{w}=2$.
	\end{example}
	
	\begin{example}
		\label{ex: projection}
		In this example we present functionals $\phi\in(\ell^{\infty})^{*}$ satisfying $\norm{\phi}_{w} = 2\norm{\phi}$, which are an analogue of the previous example. More precisely, let $m\in\N$ and $\phi_{m}\colon \ell^{\infty}\to\C$ be the projection onto the $m$-th coordinate, i.e. for all $a\in\ell^{\infty}$ we set $\phi_{m}(a) := a_{m}$. Then clearly $\norm{\phi_{m}} = 1$, and $\norm{\phi_{m}}_{w} = 2$ follows with the same approach as in the previous example.
	\end{example}
	
	Although the previous examples show how these elementary estimates might be useful in some situations, they can provide no useful information in general. For example Chien and Nakazato \cite{chien1} have shown, that the numerical radius of a weighted shift with geometric weights $a=(1,r,r^{2},r^{3},...)$, $0\leq r\leq 1$, can be nicely bounded from below by picking 'good' vectors in $\ell^{2}$, instead of using estimate \eqref{ineq: convergent}, which only provides the trivial bound $0$.
	\\\\In what follows next, we slightly generalize a result of Tam \cite[Theorem 1 (a)]{tam} stating, that for any $K>0$ and $a\in\ell^\infty$ with $|a|\leq K$ and $a_k\to K$, one has $W(T_a) = \menge{z\in\C : |z| < K}$, by relaxing the convergence condition to a condition involving the largest value a Banach limit can attain on a the sequence of absolute values of a bounded sequence. Note that we show in Remark \ref{rm: Cesaro}, that with the other assumptions of the result, it includes the case of Cesàro convergence, which in turn illustrates that this is a generalization, but actually only a small improvement of the original result, since in this case the assumptions imply that there exists a set $J\subseteq \N$ with density equal to 1 such that $|a|$ converges along $J$ to $K$. Note that Tam's proof works as before.
	
	\begin{corollary}
		Let $K>0$, $a\in\ell^{\infty}$ such that $|a|\leq K$ and 
		\begin{align}
			\label{eq: Tam}
			\lim_{n\to\infty}\sup_{j\in\N}  \frac{1}{n} \sum_{k=1}^{n} \abs{a_{k+j-1}} = K.
		\end{align}
		Then $W(T_{a}) = \menge{z\in\C\,:\,\abs{z}< K}$.
	\end{corollary}
	\begin{proof}
		From inequality \eqref{eq: estimate any seq}, we conclude
		\begin{align*}
			K = \lim_{n\to\infty}\sup_{j\in\N}  \frac{1}{n} \sum_{k=1}^{n} \abs{a_{k+j-1}} \leq \norm{a}_{w} \leq \norm{a}_{\infty} \leq K.
		\end{align*}
		Hence $\norm{T_{a}} = \norm{a}_{w} = K$ and since $W(T_{a})$ is a circular disc centered at the origin it remains to check that $W(T_{a})$ is open. Now suppose that there exists $\lambda\in\sigma_{p}(T_{a})$, $\lambda\neq 0$. Then there exists $x\in\ell^{2},x\neq 0$ such that $T_{a}x = \lambda x$, that is 
		\begin{align*}
			\lambda x_{1} = 0,\quad a_{k}x_{k} = \lambda x_{k+1},\,\forall k\in\N.
		\end{align*}
		We readily conclude $x_{k}=0$ for any $k\in\N$, so $x = 0$, which is a contradiction, so there cannot exist such a $\lambda$. So $\sigma_{p}(T_{a})\subseteq \menge{0}$. Now if $W(T_{a})$ is closed and $\lambda\in\partial W(T_{a})$, then $\abs{\lambda} = \norm{T_{a}} = K$ and by a classical result $\lambda\in \sigma_{p}(T_{a})$ (see \cite[Solution 212, p. 316]{halmos}), which is a contradiction as $K>0$. Hence $W(T_{a})$ is open. 
	\end{proof}

	\begin{remark}
		\label{rm: Cesaro}
		Let us provide some remarks on a specialization of condition \eqref{eq: Tam}, i.e. the Cesàro convergent case:
		\begin{enumerate}
			\item If we substitute in the previous Corollary condition \eqref{eq: Tam} with 
			\begin{align*}
				\limsup_{n\to\infty} \frac{1}{n}\sum_{k=1}^{n} \abs{a_{k}} = K,
			\end{align*}
			then we immediately obtain \eqref{eq: Tam}, due to the assumptions of the Corollary. Hence, the conclusion of the previous Corollary applies also in the case of Cesàro convergent $\abs{a}$ with Cesàro limit $K$.
			
			\item Note that, as Tam concluded \cite[Proof of Theorem 1 (a)]{tam}, the conditions $K>0$ and $a\in\ell^{\infty}$ with $\abs{a}\to K$, clearly imply that only finitely many entries of $|a|$ can be equal to zero and all other entries have to be bigger than zero. However the conditions $\abs{a}\leq K$ and 
			\begin{align*}
				\lim_{n\to\infty} \frac{1}{n}\sum_{k=1}^{n} |a_{k}| = K
			\end{align*}
			do not lead to such a strong conclusion. To be more precise as the sequence of Cesàro averages $\sigma_{n} := \frac{1}{n}\sum_{k=1}^{n}|a_{k}|$ converges to $K>0$, either $\sigma_{n}\neq 0$ for all $n\in\N$ or there exists some minimal index $N\in\N$ such that we have $\sigma_{1}=…=\sigma_{N}=0$ and $|\sigma_{n}|\neq 0$ for all $n\in\N,n> N$. However this implies by using the equality $(n+1)\sigma_{n+1} -n\sigma_{n} = a_{n+1}$, that $a_{1}=…=a_{N}=0$ and $a_{N+1}\neq 0$. However as $\sigma_{n}$ is an average over the entries of $a$, we do not gain information about further elements of the sequence $a$ as $a_{N+2}$ could be zero and we would still have $\sigma_{N+2}\neq 0$. To illustrate this we consider the sequence $a=(1,0,1,1,0,1,1,1,0,...)$ or more precisely $a=(a_{n})$ defined via
			\begin{align*}
				a_{n} = \begin{cases}
					0, & \mbox{if } n=\frac{l^{2}+3l}{2} \mbox{ for }l\in\N, \\
					1, & \mbox{else.}
				\end{cases}
			\end{align*}
			Then clearly $\sigma_{n}>0$, $a_{n}\leq 1$ for all $n\in\N$ and it is relatively easy to show that $\sigma_{n}\to 1$. So the previous Corollary applies and yields that $W(T_{a})$ is the open unit disk about the origin. However infinitely many entries of $a$ are equal to zero and the set of these entries has density equalt to 0, but $a$ converges along a set of density equal to 1 to it's Cesàro limit $K=1$.
			
			\item The last conclusions of the previous item are true for any sequence satisfying the following conditions: let $K>0$, $a\in\ell^{\infty}$ such that $|a|\leq K$ and
			\begin{align*}
				\lim_{n\to\infty} \frac{1}{n}\sum_{k=1}^{n} |a_{k}| = K.
			\end{align*}
			Then, we have
			\begin{align*}
				0 \leq \frac{1}{n}\sum_{k=1}^{n} |K - |a_{k}|| \leq  \frac{1}{n}\sum_{k=1}^{n} K - |a_{k}| \to 0.
			\end{align*}
			So, by the Lemma of Koopman-von Neumann \cite[Lemma 9.16]{eisner}, it follows that there exists a set $J\subseteq \N$ of density equal to 1, i.e.
			\begin{align*}
				\lim_{n\to\infty} \frac{\abs{\menge{1,...,n}\cap J}}{n} = 1,
			\end{align*}
			such that
			\begin{align*}
				\lim_{\substack{n\to\infty \\ n\in J}} |a_{n}| = K.
			\end{align*}
			Hence only finitely many points $j_{1},...,j_{m}\in J$ can satisfy $a_{j_{i}} = 0,\,i=1,…,m$, since $K>0$. Consequently, the set $M:=\menge{n\in\N\,:\,a_{n}\neq 0}$ has density equal to 1, since we have $J\backslash\menge{j_{1},…,j_{m}} \subseteq M$, from which we also obtain that the set $\menge{n\in\N\,:\, a_{n}=0}$ has density equal to 0. 
		\end{enumerate}
	\end{remark}
	
	In the last result of this section we prove that any weighted shift with non-constant periodic weight is non-convexoid, whereas any weighted shift with a constant weight is clearly convexoid.
	
	\begin{corollary}
		\label{cr: non-convexoid}
		Let $a\in\ell^{\infty}$ be periodic of period $N$. If there exists $i,j\in\menge{1,...,N}$ with $i\neq j$ such that $\abs{a_{i}}\neq \abs{a_{j}}$, then 
		\begin{align*}
			r(T_{a}) < \frac{1}{N}\sum_{k=1}^{N} \abs{a_{k}} \leq \norm{a}_{w} < \max\menge{\abs{a_{1}},...,\abs{a_{N}}}, 
		\end{align*}
		i.e. $T_{a}$ is non-normaloid and non-convexoid.
	\end{corollary}
	\begin{proof}
		Since the spectrum and the closure of the numerical range of a weighted shift are both closed discs centered about the origin, a weighted shift is convexoid if and only if the spectral radius and the numerical radius are the same. Now, since $a\in\ell^{\infty}$ is non-constant, so in particular non-zero. The constant sequence $b:=\kla{\norm{a}_{\infty}}$ is periodic of period $N$ and there exists some $n_{0}\in\N$ with $|a_{n_{0}}|<|b_{n_{0}}|$, so using \cite[Lemma 4.4 (c)]{Wu/shifts}, we obtain
		\begin{align*}
			\norm{a}_{w} < \norm{b}_{w}= \norm{a}_{\infty}
		\end{align*} 
		and in particular $T_{a}$ is non-normaloid. Further by \eqref{ineq: arithm} we have
		\begin{align*}
			\frac{1}{N}\sum_{k=1}^{N} \abs{a_{k}} \leq \norm{a}_{w}.
		\end{align*} 
		We denote $A:= \prod_{k=1}^{N} a_{k}$ and conclude by Gelfand's formula for the spectral radius
		\begin{align*}
			r(T_{a}) = \lim_{n\to\infty}\norm{T_{a}^{n}}^{\frac{1}{n}} = \lim_{k\to\infty}\norm{T_{a}^{kN}}^{\frac{1}{kN}} = \lim_{k\to\infty} \kla{A^{k}}^{\frac{1}{kN}} = A^{\frac{1}{N}} = \kla{\prod_{k=1}^{N} a_{k}}^{\frac{1}{N}}
		\end{align*}
		Hence, by the inequality of arithmetic and geometric mean we find 
		\begin{align*}
			r(T_{a}) = \kla{\prod_{k=1}^{N} a_{k}}^{\frac{1}{N}} < \frac{1}{N}\sum_{k=1}^{N} \abs{a_{k}} \leq \norm{a}_{w}.
		\end{align*}
		Hence we have shown that $T_{a}$ is non-convexoid.
	\end{proof}
	
	\begin{remark}
		As we have seen in the previous theorem a weighted shift with weight $a\in\ell^{\infty}$ is convexoid if and only if $r(T_{a}) = \norm{a}_{w}$. An operator satisfying this condition is called \textit{spectraloid} (cf. \cite[Introduction of Chapter 6, p. 150]{GustafsonRao1997}). It seems that conditions and further characterizations for weighted shifts, which guarantee that a weighted shift is (non-)spectraloid or (non-)normaloid, apart from the ones for arbitrary operators on complex Hilbert spaces (cf. \cite[Section 6.1, 6.2, 6.3]{GustafsonRao1997} and the references therein) are not known in the literature up to now. We would like to point out, that thanks to \eqref{eq: sharpened connection}, if $a\in\ell^{\infty}$ satisfies one of the conditions
		\begin{align*}
			r(T_{a}) < M(\abs{a})\quad \mbox{ or }\quad M(\abs{a}) < \norm{a}_{w},
		\end{align*}
		it follows that $T_{a}$ is non-spectraloid and hence non-convexoid. However, it is not clear how to verify either of these conditions in general. Thus it would be interesting to find other conditions or characterizations, which might be easier to verify.
	\end{remark}

	\section{Norms of multiplication operators and weighted shifts on $\ell^{\infty}$}
	
	Although we have already shown in Example \ref{ex: norm of op} that in general the operator norm of an operator $R\in\mathcal{B}(\ell^{\infty})$ with respect to the uniform norm and the numerical radius of weighted shifts (in both the domain and the image) might differ, we will see in this chapter, that this difference does not appear for multiplication operators and weighted shifts on $\ell^{\infty}$.
	To avoid confusion with the forward shift $T$ and the backward shift $S$ on $\ell^{2}$, we denote by $F\in \mathcal{B}\kla{\ell^{\infty}}$ the forward shift and by $B\in \mathcal{B}\kla{\ell^{\infty}}$ the backward shift on $\ell^{\infty}$. Further, we denote for $d\in\ell^{\infty}$ by $m_{d}\in\mathcal{B}\kla{\ell^{2}}$ and $M_{d}\in\mathcal{B}\kla{\ell^{\infty}}$ the multiplication operators corresponding to $d$ on $\ell^{2}$ and $\ell^{\infty}$ respectively. Finally, we denote by $F_{d}\in\mathcal{B}\kla{\ell^{\infty}}$ and $B_{d}\in\mathcal{B}\kla{\ell^{\infty}}$ the (unilateral) weighted forward- and backward-shift on $\ell^{\infty}$ associated to $d$ respectively. Again we will mean by 'shifts' only forward-shifts.
	
	As we have seen in the introductory section, a first estimate for the operator norms of the above operators with respect to $\norm{\cdot}_{w}$ follows from \eqref{ineq: op norm} and we have
	\begin{align*}
		\frac{\norm{d}_{\infty}}{2} \leq \norm{M_{d}}_{w},\,\norm{F_{d}}_{w},\norm{B_{d}}_{w} \leq 2\norm{d}_{\infty},\quad \forall d\in\ell^{\infty}.
	\end{align*}
	This naturally raises the question if the constant $\frac{1}{2}$ and 2 are are optimal in these inequalities. The following Theorem shows that this is actually not the case, as we already claimed above.
	
	\begin{theorem}
		Let $d\in\ell^{\infty}$. Then we have $\norm{M_{d}}_{w} = \norm{F_{d}}_{w} = \norm{B_{d}}_{w} = \norm{d}_{\infty}$.
	\end{theorem}
	\begin{proof}
		We compute the norm of each operator separately.
		\begin{enumerate}
			\item[(a)] 
			For $a\in\ell^{\infty}$ we have
			\begin{align*}
				|(M_{d}a)_{k}| = |d_{k}a_{k}| \leq \norm{d}_{\infty}\cdot |a_{k}|,
			\end{align*}
			so $|M_{d}a|\leq \norm{d}_{\infty}\cdot |a|$. Hence we find by monotonicity
			\begin{align*}
				\norm{M_{d}a}_{w} = \norm{|M_{d}a|}_{w} \leq \norm{d}_{\infty} \norm{a}_{w}.
			\end{align*}
			It remains to prove equality. To do so, fix $k\in\N$ and let $e_{k}\in\ell^{\infty}$ be the $k$-th standard unit vector in $\ell^{\infty}$. One can easily verify that $\norm{2e_{k}}_{w} = 1$. Further we conclude from inequality \eqref{ineq: estimate nr} with $n=2$ and $l=1$ that
			\begin{align*}
				\norm{M_{d}2e_{k}}_{w} & = \norm{2d_{k}e_{k}}_{w} = \abs{2d_{k}} \norm{e_{k}}_{w} \overset{\eqref{ineq: estimate nr}}{\geq } \frac{\abs{2d_{k}}}{2} = \abs{d_{k}}.
			\end{align*}
			Finally this yields together with the first inequality
			\begin{align*}
				\norm{d}_{\infty} = \sup_{k\in\N} \abs{d_{k}} \leq \sup_{k\in\N} \norm{M_{d}2e_{k}}_{w} \leq  \norm{d}_{\infty},
			\end{align*}
			which proves the claim for $M_{d}$.
			
			\item[(b)] We check the claim for $B_{d}$. But first, let $z\in W(T_{Ba})$. Then there exists $x\in\ell^{2}$ with $\norm{x}_{\ell^{2}}=1$ such that $z=\scalar{T_{Ba}x,x}_{\ell^{2}}$. By using $ST=Id$ and $S=T^{*}$ we find
			\begin{align*}
				z = \scalar{T_{Ba}x,x}_{\ell^{2}} & = \sum_{k=1}^{\infty} a_{k+1}x_{k}\overline{x_{k+1}} = \scalar{m_{a}Tx,x}_{\ell^{2}} = \scalar{Tm_{a}Tx,Tx}_{\ell^{2}} = \scalar{T_{a}Tx,Tx}_{\ell^{2}}
			\end{align*}
			Since $\norm{Tx}_{\ell^{2}} = \norm{x}_{\ell^{2}} = 1$ we have $z\in W(T_{a})$, so $\norm{Ba}_{w} \leq \norm{a}_{w}$. Hence we find
			\begin{align*}
				\norm{B_{d}a}_{w} = \norm{d\cdot Ba}_{w} = \norm{|d\cdot Ba|}_{w} = \norm{|d|\cdot |Ba|}_{w} \leq \norm{d}_{\infty}\cdot \norm{B|a|}_{w} \leq \norm{d}_{\infty} \norm{a}_{w}.
			\end{align*}
			Again one can prove the equality similar as for $M_{d}$.
			
			\item[(c)]  We check the claim for $F_{d}$. For $a\in\ell^{\infty}$ and $x\in\ell^{2}$ we have
			\begin{align*}
				m_{F_{d}a}x = \kla{0,d_{1}a_{1}x_{2},d_{2}a_{2}x_{3},...} = T_{da}Sx = T_{M_{d}a}Sx
			\end{align*}
			Hence we have
			\begin{align*}
				T_{F_{d}a} = Tm_{F_{d}a} = TT_{M_{d}a}S,   
			\end{align*}
			which yields due to $T^{*} = S$
			\begin{align}
				\label{eq:equal 1}
				\scalar{T_{F_{d}a}x,x}_{\ell^{2}} = \scalar{T_{M_{d}a}Sx,Sx}_{\ell^{2}},\quad \forall x\in\ell^{2}.
			\end{align}
			Thus we find
			\begin{align*}
				\abs{\scalar{T_{F_{d}a}x,x}} & \overset{\eqref{eq:equal 1}}{\leq} w\kla{T_{M_{d}a}} \norm{Sx}_{\ell^{2}}^{2} \leq \norm{M_{d}a}_{w}\norm{x}_{\ell^{2}}^{2} \leq \norm{d}_{\infty} \norm{a}_{w}\norm{x}_{\ell^{2}}^{2} ,
			\end{align*}
			from which we conclude that $\norm{F_{d}a}_{w}\leq \norm{d}_{\infty}\norm{a}_{w}$, so $\norm{F_{d}}_{w}\leq\norm{d}_{\infty}$. Again a similar argument as for $M_{d}$ shows equality.
			
		\end{enumerate}
		
	\end{proof}
	
	\section{The numerical radius of weighted shifts and applications to Banach limits}
	
	Recall that $\mathbb{1}=(1,1,...)\in\ell^{\infty}$ and note that for $a\in\ell^{\infty}$ the inequality $a \geq 0$ means that for all $k\in\N$ we have $a_{k}\geq 0$. Further we denote by $c$ the complex Banach space of all complex-valued bounded sequences. Banach limits are normalized, shift-invariant linear functionals, which continuously extend the limit functional from $c$ to $\ell^{\infty}$ with respect to the uniform norm, i.e.
	\begin{definition} \textbf{(Complex-valued Banach limits)}
		\label{def:banach limit}
		A linear functional $L:\ell^{\infty}\rightarrow \C$ is called \textit{(complex-valued) Banach limit}, if it satisfies the following conditions
		\begin{enumerate}
			\item $L(a) = L(Ba),\,\forall a\in\ell^{\infty}$.
			\item $L(a)\geq 0,\,\forall a\in\ell^{\infty},\,a\geq 0$.
			\item $L(\mathbb{1}) = 1$ and $\abs{L(a)} \leq \norm{a}_{\infty},\,\forall a\in\ell^{\infty}$.
			\item $L(a) = \lim\limits_{n\to\infty} a_{n},\,\forall a\in c$.
		\end{enumerate}
	\end{definition}
	Banach limits are continuous linear extensions of the limit operator from the Banach space of convergent sequences $(c,\norm{\cdot}_{\infty})$ onto the Banach space $(\ell^{\infty},\norm{\cdot}_{\infty})$ of bounded sequences. Yet, it is also possible to construct Banach limits as continuous linear extensions of the Cesàro limit operator on the closed subspace of Cesàro convergent sequences. In any case the existence of Banach limits is a consequence of the Hahn-Banach extension Theorem. For proofs we refer the reader to \cite[III. §7, p. 82]{conway} and \cite[Theorem, p. 309]{Sucheston}. Note that Sucheston proved in addition, that the maximal value of a real-valued Banach limit on a sequence $a\in \ell_{\R}^{\infty}$ is given by
	\begin{align*}
		M(a) := \lim_{n\to\infty} \sup_{j\in\N} \frac{1}{n}\sum_{k=1}^{n} a_{n+j-1}.
	\end{align*}
	
	Further the existence of real-valued Banach limits gives rise to complex-valued Banach limits. A classical proof of this fact can be again found in Conway's book \cite[III. §7, p. 83]{conway}. 
	
	The general theory of Banach limits has connections to the existence invariant measures and densities. For recent developments we refer to \cite{semenov/sukochev},\cite{Usachev} and \cite{sofi} as well as the references therein. Further we would like to stress that from a measure theoretical perspective, Banach limits are related to 'pure' finitely additive measures on the natural numbers, which are concentrated at $'\infty'$, i.e. they are not $\sigma$-additive and define measures lying in between the upper and lower density \cite[Example 2.1.3 (9)]{RaoRao1983}.
	
	Any complex-valued Banach limit $L:\ell^{\infty}\to\C$ clearly satisfies $\norm{L}=1$. From this we conclude as a first estimate for its norm with respect to $\norm{\cdot}_{w}$, that $\norm{L}_{w}\leq 2$ and one could expect that $\norm{L}_{w} = 2$, since we have already shown in Example \ref{ex: projection} that very simple functionals satisfy this equality. However from \eqref{eq: estimate any seq} and the fact that $M(a)$ is the largest value a real Banach limit can attain on a given sequence $a\in\ell_{\R}^{\infty}$, we immediately get the following:
	\begin{corollary}
		\label{cr: Banach limits}
		Any real-valued Banach limit $L:\ell_{\R}^{\infty}\to \R$ satisfies for all $a\in\ell_{\R}^{\infty}$ the norm estimate
		\begin{align*}
			\abs{L(a)}\leq \norm{a}_{w}.
		\end{align*}
	\end{corollary}

	Now adopting the before mentioned construction of complex Banach limits with respect to $\norm{\cdot}_{\infty}$ from Conway's book \cite[III. §7 p. 83]{conway}, one can obtain the same inequality for complex-valued Banach limits, by following the proof therein and applying the previous corollary. However, we present the full proof here for the convenience of the reader as we explicitly need the monotonicity of $\norm{\cdot}_{w}$ to conclude the proof.

	\begin{theorem} 
		\label{th: complex banach lim}
		There exists a linear functional $L:\ell^{\infty}\rightarrow \C$ such that
		\begin{enumerate}
			\item $L(a) = L(Ba),\,\forall a\in\ell^{\infty}$.
			\item $L(a)\geq 0,\,\forall a\in\ell^{\infty},\,a\geq 0$.
			\item $L(\mathbb{1}) = 1$ and $\abs{L(a)} \leq \norm{a}_{w},\,\forall a\in\ell^{\infty}$.
			\item $L(a) = \lim\limits_{n\to\infty} a_{n},\,\forall a\in c$.
		\end{enumerate}
	\end{theorem}
	\begin{proof}
		We can split each $a\in\ell^{\infty}$ into unique sequences $\mathrm{Re}(a) := \kla{\mathrm{Re}(a_{k})}\in\ell_{\R}^{\infty}$ and $\mathrm{Im}(a) := \kla{\mathrm{Im}(a_{k})}\in\ell_{\R}^{\infty}$ such that we have 
		$a=\mathrm{Re}(a)+i\,\mathrm{Im}(a)$. Let $\tilde{L}:\ell_{\R}^{\infty}\to \R$ be any real-valued Banach limit. Then the map $L:\ell^{\infty}\to\C$ which we define via 
		\begin{align*}
			L(a) := \tilde{L}(\mathrm{Re}(a)) + i\tilde{L}(\mathrm{Im}(a)),\quad a\in\ell^{\infty}
		\end{align*}
		is a linear map satisfying (1), (2), (4) and $L(\mathbb{1}) = 1$, so it remains to prove the norm estimate in (3). For this purpose let $E_{1},...,E_{m}\subseteq \N$ be pairwise disjoint subsets of $\N$ and $a_{1},...,a_{m}\in\C$. We denote by $\chi_{k}$ the characteristic sequences of $E_{k}$, $k=1,...,m$ and consider a sequence $a := \sum\limits_{k=1}^{m} a_{k}\chi_{k}\in\ell^{\infty}$. We will call such a sequence from now on \textit{step sequence} in analogy to step functions. Then we have
		\begin{align*}
			L(a) = \sum_{k=1}^{m} a_{k}\tilde{L}\kla{\chi_{k}},
		\end{align*}
		as $0\leq \chi_{k}\leq \mathbb{1}$ and so by definition and (b) we have $0\leq \tilde{L}(\chi_{k})=L(\chi_{k})\leq 1$. Assuming that $\norm{a}_{w}\leq 1$ we hence conclude from Corollary \ref{cr: Banach limits}
		\begin{align*}
			\abs{L(a)} & = \abs{\sum_{k=1}^{m} a_{k}\tilde{L}\kla{\chi_{k}}} \leq \sum_{k=1}^{m} \abs{a_{k}}\tilde{L}\kla{\chi_{k}} = \sum_{k=1}^{m} \tilde{L}\kla{\abs{a_{k}}\chi_{k}} = \tilde{L}\kla{\sum_{k=1}^{m} \abs{a_{k}}\chi_{k}} = \tilde{L}(\abs{a})   \leq \norm{a}_{w} \leq 1.
		\end{align*} 
		Next we apply the standard result (see Lemma \ref{lm: monotonic approx}), that for each $a\in\ell^{\infty}$ there exists a sequence of step sequences $\kla{a^{(n)}}\subseteq\ell^{\infty}$ such that $\norm{a-a^{(n)}}_{w}\to 0$ and $\abs{a^{(n)}}\leq \abs{a}$ for all $n\in\N$. Now given an element $a\in\ell^{\infty}$ with $\norm{a}_{w}\leq 1$ and such a sequence of step sequences, the monotonicity of the numerical radius of weighted shifts yields  $\norm{a^{(n)}}_{w}\leq 1$ for all $n\in\N$. From this we conclude together with the continuity of $L$ and the previous estimate, that (c) holds.   
	\end{proof}
	
	As all complex Banach limits arise from real-valued Banach limits, it follows, that 
	
	\begin{corollary}
		\label{cr: compl. Banach limits}
		For any complex-valued Banach limit $L:\ell^{\infty}\to \C$, we have $\norm{L}_{w} = 1$.
	\end{corollary}
	
	Finally, in the last result of this section, we present based on inequality \eqref{eq: sharpened connection}, necessary and sufficient conditions for a sequence $a\in\ell^{\infty}$ with $\norm{a}_{\infty}\leq 1$ such that all Banach limits $L:\ell^{\infty}\to\C$ satisfy $\abs{L(a)}<1$. 
	\begin{theorem}
		\label{th: uniform BL-theorem}
		Let $a\in\ell^{\infty}$ with $\norm{a}_{\infty} = 1$. Then the following are equivalent
		\begin{enumerate}
			\item $r(T_{a})<1$
			\item $M(\abs{a})<1$
			\item $\norm{a}_{w}<1$
			\item there exist $C>0$ and $\lambda\in(0,1)$ such that
			\begin{align*}
				\prod_{k=1}^{n} \abs{a_{k+j-1}} \leq C\lambda^{n},\quad \forall n\in\N,\,j\in\N.
			\end{align*}
		\end{enumerate}
	\end{theorem}
	The proof requires the following auxiliary result (for a proof see \cite[Proof of Theorem 1.3-2, p. 10]{rao}).
	\begin{theorem}
		\label{th: numerical range}
		Let $H$ be a complex Hilbert space and $T\in\mathcal{B}(H)$. If $\lambda\in \overline{W(T)}$ with $\abs{\lambda} = \norm{T}$, then $\lambda\in \sigma_{ap}(T)$, where $\sigma_{ap}$ is the approximate point spectrum, i.e.
		\begin{align*}
			\sigma_{ap}(T) = \menge{\lambda\in\C\,:\,\exists \kla{x_{n}}\subseteq H, \norm{x_{n}} = 1 \mbox{ such that } \norm{(\lambda-T)x_{n}}\to 0}. 
		\end{align*}
	\end{theorem}
	
	\begin{proof}[Proof of Theorem 4.6]
		From \eqref{eq: sharpened connection}, we have for any $a\in\ell^{\infty}$
		\begin{align*}
			r(T_{a}) \leq M(\abs{a}) \leq \norm{a}_{w} \leq \norm{a}_{\infty}.
		\end{align*}
		Hence (3) $\Rightarrow$ (2) $\Rightarrow$ (1). Further, as $\norm{T_{a}} = \norm{a}_{\infty} = 1$, by Theorem \ref{th: numerical range}, we have $\overline{W(T_{a})}\cap\T\subseteq\sigma(T_{a})\cap\T$, where $\T$ denotes the unit circle in $\C$. The converse inclusion is true as $\sigma(T_{a}) \subseteq \overline{W(T_{a})}$, so $\overline{W(T_{a})}\cap\T = \sigma(T_{a})\cap\T$. As both the spectrum and the numerical range of weighted shifts are circular disks about the origin, it follows that (1) $\Rightarrow$ (3). Thus, we have (1) $\Leftrightarrow$ (2) $\Leftrightarrow$ (3). 
		\\It remains to show that (1) $\Leftrightarrow$ (4). If (4) holds, we have
		\begin{align*}
			r(T_{a}) = \lim_{n\to\infty} \sup_{j\in\N} \kla{\prod_{k=1}^{n}\abs{a_{k+j-1}}}^{\frac{1}{n}} \leq \lim_{n\to\infty} C^{\frac{1}{n}}\lambda = \lambda < 1.
		\end{align*}
		If (1) holds, we have
		\begin{align*}
			\lim_{n\to\infty} \norm{T_{a}^{n}}^{\frac{1}{n}} = r(T_{a}) < 1. 
		\end{align*}
		Hence, for all $\lambda \in (r(T_{a}),1)$ there exists some $N\in\N$ such that for all $n\in\N$ with $n\geq N$, we have
		\begin{align*}
			\norm{T_{a}^{n}}^{\frac{1}{n}} < \lambda.
		\end{align*}
		Choosing 
		\begin{align*}
			C:=\max\menge{1,\lambda^{-1}\norm{T_{a}},\lambda^{-2}\norm{T_{a}^{-2}},…,\lambda^{-N}\norm{T_{a}^{-N}}}+1,
		\end{align*}
		we find that
		\begin{align*}
			\norm{T_{a}^{n}} \leq C\lambda^{n},\quad \forall n\in\N.
		\end{align*}
		Now as 
		\begin{align*}
			\norm{T_{a}^{n}} = \sup_{j\in\N} \prod_{k=1}^{n} \abs{a_{k+j-1}},
		\end{align*}
		(4) follows.
	\end{proof}

	\begin{example}
		A sequence with $\norm{a}_{\infty} =1$, which does not satisfy any of the conditions of Theorem \ref{th: uniform BL-theorem} is given by $a=(1,0,1,1,0,1,1,1,0,...)$. This is the example we looked at in the second item of Remark \ref{rm: Cesaro}. For this sequence, we have
		\begin{align*}
			1 = \lim_{n\to\infty}\frac{1}{n}\sum_{k=1}^{n} \abs{a_{k}} \leq M(\abs{a}) \leq \norm{a}_{w} \leq 1, 
		\end{align*} 
		so $\norm{a}_{w} = M(\abs{a}) = 1$. Further, one can find for any $n\in\N$ a number $j\in\N$ such that
		\begin{align*}
			\prod_{k=1}^{n} \abs{a_{k+j-1}} = 1,
		\end{align*}
		since the sequence $a$ contains for any $n\in\N$ a block of length $n$ with all elements equal to 1. Hence, condition (4) of Theorem \ref{th: uniform BL-theorem} is not satisfied and we have $r(T_{a})=1$.
	\end{example}
	
	We denote by $\partial B_{1}^{\infty} := \menge{a\in\ell^{\infty}\,:\,\norm{a}_{\infty}=1}$ the standard unit sphere in $\ell^{\infty}$. In the last Theorem of this subsection we show that the set of sequences in $\partial B_{1}^{\infty}$, to which the last theorem applies, is dense in $\partial B_{1}^{\infty}$. Therefore set
	\begin{align*}
		D := \menge{a\in\ell^{\infty}\,:\,\exists C>0,\lambda\in\kla{0,1}\mbox{ with }\prod_{k=1}^{n}\abs{a_{k+j-1}}\leq C\lambda^{n},\,\forall n\in\N,j\in\N}
	\end{align*}
	and $D_{1} := D\cap \partial B_{1}^{\infty}$. Then we have
	\begin{theorem}
		$D_{1}$ is dense in $\partial B_{1}^{\infty}$.
	\end{theorem}
	\begin{proof}
		In order to prove the second statement we fix $a\in \partial B_{1}^{\infty}$. Now there exists a subsequence $\kla{a_{k_{l}}}\subseteq \kla{a_{k}} = a$ such that $\abs{a_{k_{l}}}\to 1$. W.l.o.g. we may assume that $k_{l+1}-k_{l}\geq 2$ for all $l\in\N$ (otherwise we may pass to a subsequence). We define a sequence $(b_{m})\subseteq \ell^{\infty}$ by setting for all $m\in\N$ and $k\in\N$
		\begin{align*}
			(b_{m})_{k} := \begin{cases}
				a_{k}, & \mbox{if } k=k_{l}\mbox{ for some } l\in\N, \\
				\frac{m-1}{m}\cdot a_{k}, & \mbox{else.}
			\end{cases} 
		\end{align*}
		It is not hard to check that $\lim\limits_{m\to\infty}\norm{b_{m}-a}_{\infty} = 0$ and for all $m\in\N$ one has $b_{m}\in D_{1}\cap \partial B_{1}^{\infty}$ with $\lambda_{m} = \sqrt{\frac{m-1}{m}}$ and $C_{m}\geq\sqrt{\frac{m}{m-1}}$.
	\end{proof}
	
	\begin{remark}
		Note that as a consequence of Sucheston's Theorem \cite[Theorem, p. 309]{Sucheston} and Theorem \ref{th: uniform BL-theorem}, the set $D_{1}$ is the largest set in $\partial B_{1}^{\infty}$ such that for all Banach limits $L:\ell^{\infty}\to\C$, we have
		\begin{align*}
			\abs{L(a)}<1, \quad \forall a\in D_{1}. 
		\end{align*}
	\end{remark}

	\section{Some sufficient conditions for a norm on $\ell^{\infty}$ to satisfy Corollary \ref{cr: compl. Banach limits}}
	
	In this section we summarize conditions for an arbitrary norm on $\ell^{\infty}$, which ensure that the operator norm of any complex Banach limit $L:\ell^{\infty}\to\C$ with respect to the given norm is equal to 1. By Sucheston's Theorem it is clear that a necessary condition for a norm $\norm{\cdot}:\ell^{\infty}\to\left[0,\infty\right)$ to satisfy the statements of Corollary \ref{cr: Banach limits} and Corollary \ref{cr: compl. Banach limits} is given by  
	
	\begin{enumerate}
		\item[(A1)] 	$\lim\limits_{n\to\infty}\sup\limits_{j\in\N}  \frac{1}{n} \sum\limits_{k=1}^{n} \abs{a_{k+j-1}}\leq \norm{a}, \quad \forall a\in\ell^{\infty}.$
	\end{enumerate}
	Further, in order to obtain the same topology as the topology of the uniform norm, we require that
	\begin{enumerate}
		\item[(A2)] There exists a constant $C\geq1$ such that for all $a\in\ell^{\infty}$, we have
		\begin{align*}
			\quad \norm{a}\leq \norm{a}_{\infty} \leq C\norm{a}.
		\end{align*}
	\end{enumerate}	
	Hence (A2) is just saying that $\norm{\cdot}$ is equivalent to the uniform norm on $\ell^{\infty}$, with the equivalence constants $C\geq1$ and $c=1$. The latter ensures, that all Banach limits satisfy $\norm{L}=1$. 
	
	Having these two conditions at hand we readily conclude the following result by adopting the proof of Theorem \ref{th: complex banach lim}:
	\begin{theorem}
		Let $\norm{\cdot}:\ell^{\infty}\to[0,\infty)$ be a monotonic norm satisfying \emph{(A1)} and \emph{(A2)} and $a\in\ell^{\infty}$. Then all complex-valued Banach limits $L:\ell^{\infty}\rightarrow \C$ satisfy for all $a\in\ell^{\infty}$ the norm estimate
		\begin{align*}
			\abs{L(a)} \leq \norm{a}.
		\end{align*}
	\end{theorem}
	
	Next we present a sufficient condition for a norm on $\ell^{\infty}$ to satisfy (A1).
	\begin{definition}
		We say a norm $\norm{\cdot}:\ell^{\infty}\to\left[0,\infty\right)$ satisfies a \textit{shift-invariant arithmetic mean type estimate}, which we abbreviate by (S-IAM), if there exist $l\in\N$ and $p\in\Z$ such that for all $n\in\N$ with $n\geq l+\abs{p}$ and $j\in\N$ the inequality
		\begin{align*}
			\sup_{n\in\N} \frac{1}{n}\abs{\sum_{k=l}^{n-p}a_{k+j-1}}\leq \norm{a},\quad \forall a\in\ell^{\infty}.
		\end{align*}
	\end{definition}
	
	\begin{example}
		Of course the uniform norm satisfies (S-IAM) with $l=1$ and $p=0$ and $\norm{\cdot}_{w}$ satisfies (S-IAM) with $l=1$ and $p=1$ as we have shown in \eqref{ineq: estimate nr}. Further, it can be shown (which we do not prove here), that for any $m\in\N$ the map $\norm{\cdot}_{m}\colon\ell^{\infty}\to[0,\infty)$ given by
		\begin{align*}
			\norm{a}_{m} := w(T^{m-1}T_{a}),\quad a\in\ell^{\infty},
		\end{align*} 
		is a monotonic norm on $\ell^{\infty}$, that satisfies (A2). Moreover for all $n\in\N$ with $n\geq m$ and each $j\in\N$, we have
		\begin{align*}
			\frac{1}{n}\sum_{k=1}^{n-m}|a_{k+j-1}| \leq \norm{a}_{m},\quad \frac{1}{n}\sum_{k=m+1}^{n} |a_{k+j-1}|\leq \norm{a}_{m}.
		\end{align*}
		So each $\norm{\cdot}_{m}$ satisfies (S-IAM) with $l=0$ and $p=-m$ or $l=m+1$ and $p=0$. Hence there are other examples of norms satisfying the  conditions (S-IAM) and (A2).
	\end{example}
	
	Finally we show that an absolute norm (hence also a monotonic norm) satisfying (S-IAM) satisfies (A1). 
	
	\begin{theorem}
		Let $\norm{\cdot}:\ell^{\infty}\to[0,\infty)$ be an absolute norm satisfying \emph{(S-IAM)}. Then for any $a\in\ell^{\infty}$, we have
		\begin{align*}
			\lim_{n\to\infty}\sup_{j\in\N}  \frac{1}{n} \sum_{k=1}^{n} \abs{a_{k+j-1}} \leq \norm{a}.
		\end{align*}
		Further items (a) - (d) of Corollary \ref{cr: estimates} hold similarly for the norm $\norm{\cdot}$.
	\end{theorem}
	\begin{proof}
		As in the proof of Theorem \ref{th: estimates} we use the strategy of Sucheston \cite{Sucheston}. Consider again for any $n\in\N$ the linear map $C_{n}:\ell^{\infty}\to\ell^{\infty}$ given by
		\begin{align*}
			C_{n}a := \kla{\frac{1}{n} \sum_{k=1}^{n} a_{k+j-1} }_{j\in\N},\quad a\in\ell^{\infty}.
		\end{align*}
		and recall that the sequence $\kla{c_{n}}$ given by
		\begin{align*}
			c_{n}:=\norm{C_{n}\abs{a}}_{\infty}
		\end{align*}
		is convergent. Further for all $n\in\N$ with $n\geq l+\abs{p}$, where $p\in\Z$ and $l\in\N$ come from (A1), and for all $j\in\N$, we have 
		\begin{align*}
			\frac{1}{n} \abs{\sum_{k=1}^{n} a_{k+j-1}} \overset{\mbox{\tiny(S-IAM)}}{\leq} \norm{\abs{a}} + \frac{1}{n} \kla{\sum_{k=1}^{l-1} \abs{a_{k+j-1}} + \sum_{k=\min\menge{n,n-p+1}}^{\max\menge{n,n-p+1}} \abs{a_{k+j-1}}}\leq \norm{\abs{a}} + \frac{\norm{a}_{\infty}\kla{l-1+\abs{p}}}{n},   
		\end{align*}
		from which we obtain
		\begin{align*}
			\lim_{n\to\infty}\sup_{j\in\N}  \frac{1}{n} \sum_{k=1}^{n} \abs{a_{k+j-1}} = \lim_{n\to\infty} c_{n} \leq \norm{\abs{a}} = \norm{a}.
		\end{align*}
	\end{proof}

	\appendix
	\section{}
	
	The following Lemma is a standard result for functions $f\in L^{\infty}(\Omega,\mu)$ on a measure space $(\Omega,\mathcal{A},\mu)$.  However, we only state it for sequences, i.e. in the case where $\Omega = \N$, $\mathcal{A}=\mathcal{P}(\N)$ is the power set of $\N$ and $\mu = \abs{\cdot}$ is the counting measure. 
	
	\begin{lemma}
		\label{lm: monotonic approx}
		Let $\norm{\cdot}:\ell^{\infty}\to[0,\infty)$ be norm satisfying \emph{(A2)} and $a\in\ell^{\infty}$. Then there exists a sequence $a^{(n)}\in\ell^{\infty}$ of step sequences such that $\norm{a^{(n)}-a}\to0$ and $\abs{a^{(n)}}\leq \abs{a}$, i.e.
		\begin{align*}
			\abs{a_{k}^{(n)}}\leq \abs{a_{k}},\quad \forall k\in\N.
		\end{align*}
	\end{lemma}
	
\end{document}